\documentclass[reqno,a4paper]{amsart}
    \usepackage{amsmath, amsfonts, amssymb, amsthm, mathtools, etoolbox,
            enumitem, url, tikz-cd, centernot, xcolor}
    \usepackage{centernot}
\usepackage{makecell}
\usepackage{xcolor}
\usepackage{dsfont}
\usepackage{url}
\usepackage[all]{xy}
\usepackage[utf8]{inputenc}
\usepackage{wrapfig,textcomp,csquotes}
\usepackage[colorlinks=true, linkcolor=blue, citecolor=blue, urlcolor=blue, breaklinks=true]{hyperref}
\usepackage[capitalise]{cleveref}
\usepackage{todonotes}
\usetikzlibrary{positioning}
\usetikzlibrary{shapes,arrows.meta,calc}
\usetikzlibrary{arrows}

\newtheorem{theorem}{Theorem}[subsection]

\newtheorem{Theorem}[theorem]{Theorem}
\theoremstyle{definition}
\newtheorem{definition}[theorem]{Definition}
\newtheorem{remark}[theorem]{Remark}
\newtheorem{example}[theorem]{Example}
\newtheorem{corollary}[theorem] {Corollary}
\newtheorem*{acknow}{\textup{Acknowledgement}}

\newcolumntype{x}[1]{>{\centering\arraybackslash}p{#1}}

\numberwithin{equation}{section}

\newcommand{\id}{\mathrm{id}}

\newcommand{\AffiliationsBlock}{%
  \vspace{0.8em}
  \begingroup
  \centering
  \footnotesize\itshape
  $^{a}$ Department of Mathematics, Indian Institute of Technology Madras, \\ 
  Chennai-600036, Tamil Nadu, India.\\\vspace{2mm}
  $^{b}$ The Institute of Mathematical Sciences, A CI of Homi Bhabha National Institute,\\
  Chennai-600113,  Tamil Nadu, India.\par
  \endgroup
  \vspace{-1.3em}
}\makeatletter
\patchcmd{\@maketitle}{%
  \ifx\@empty\@dedicatory
  \else
}{%
  \AffiliationsBlock
  \vspace{1.5em}
  \ifx\@empty\@dedicatory
  \else
}{}{\message{Patch failed!}}
\makeatother

\begin{document}
\title{A Precise Treatment of Soft Quotient Topology and Soft Covering Maps}
\author[S. Mandal]{Souvik Mandal$^{a,*}$}
\author[A. Sarkar]{Ankur Sarkar$^{b}$}
\thanks{\hspace{-1.1em}* Corresponding author.}
\thanks{\raggedright\hspace{0.53em}
  \makebox[8.5em][l]{\textit{E-mail addresses:}}%
  \begin{minipage}[t]{0.75\textwidth}
    \texttt{ma22d014@smail.iitm.ac.in}, \texttt{ssouvik.xyz@gmail.com} (S. Mandal); \\
    \texttt{ankurimsc@gmail.com} (A. Sarkar).
  \end{minipage}}


\begin{abstract}
    We develop a foundational theory of soft quotient topology, providing a systematic approach to quotient constructions in soft topological spaces. We establish the universal property of soft quotient topology and investigate the relationship between global soft topology and its parametric slices, noting that slice-wise quotient behaviour is not sufficient to characterise soft quotients. The usefulness of the framework is illustrated through the construction of soft quotient spaces, together with a study of soft group actions and their orbit spaces. We propose a precise definition of soft covering maps that resolves inconsistencies found in the existing literature. Finally, to illustrate the applicability of our framework, we discuss a potential application in multi-agent motion planning, showing how it can significantly reduce combinatorial complexity under parametric uncertainty.
\end{abstract}
\maketitle
\vspace{-0.5em}
\begin{center}
\begin{minipage}{0.845\textwidth}
    \footnotesize
    \begin{list}{}{%
        \leftmargin=5.5em 
        \labelwidth=5.5em
        \labelsep=0pt \parsep=0pt \topsep=0pt \itemsep=0pt
    }
        \item[\textit{Keywords:}\hfill] Soft Topology, Soft Quotient Space, Soft Adjunction Space,  Soft Group action, Soft Covering map, Motion Planning.
    \end{list}

    \vspace{4pt} 

    \begin{list}{}{%
        \leftmargin=18.5em 
        \labelwidth=18.5em
        \labelsep=0pt \parsep=0pt \topsep=0pt \itemsep=0pt
    }
        \item[2020\hspace{1mm}\textit{Mathematics Subject Classification:}\hfill] Primary 54A40, 57M10;\\ Secondary 03E72, 54B15, 54B17, 54C10.
    \end{list}
\end{minipage}
\end{center}
\vspace{1em}
\section{Introduction}
Soft set theory, introduced by Molodtsov \cite{Molodtsov1999SoftSets}, provides a parameter-dependent framework for modelling uncertainty, where each parameter determines a classical subset of a universe. Building on this idea, soft topology has been developed by several authors, notably Shabir and Naz \cite{ShabirNaz2011SoftTopo}, Cagman, Karatas and Enginoglu \cite{agman2011SoftT}. These works establish the basic notions of soft open sets, soft continuity, and slice topologies. A fundamental feature of soft topology is that global properties of a soft space need not be reflected at the level of individual slices, motivating a genuinely parametric reformulation of classical topological concepts. 

In classical topology, quotient maps arise naturally in the construction of adjunction spaces, cones, suspensions, and orbit spaces, and admit several equivalent characterisations, including descriptions via universal properties and conditions involving open sets \cite{MunkresTopology,armstrong2013basic}. Despite their importance, a systematic theory of quotient maps in soft topology has not been developed. The only related work is the notion of fuzzy soft quotient topology introduced in \cite{Haripamyu2022FuzzySoftQuotient}, which is based on fuzzy soft sets and slicewise membership functions, and does not address parametric maps, universal properties, or geometric constructions in the soft topological setting. 

The aim of this paper is to develop a comprehensive theory of \emph{soft quotient maps} and the associated \emph{soft quotient topology}. We establish a soft analogue of the universal mapping property and investigate the relationship between soft quotient maps and their slice restrictions. In particular, we show that slicewise quotient behaviour does not, in general, determine soft quotient behaviour, and we identify slice-determined soft topologies as the precise setting in which these notions coincide.

This framework enables the construction of soft analogues of classical quotient spaces, including soft gluing spaces, cones, and suspensions, and reveals parameter-dependent phenomena such as the appearance of distinct classical homotopy types on different slices arising from a single soft quotient. 

We introduce the notion of soft group actions and analyze the associated soft orbit spaces. We show that canonical orbit projections are soft open maps and introduce the notion of \emph{properly soft discontinuous action} of a group on a soft topological space. We examine its relationship with classical proper discontinuity on slices and obtain sufficient conditions ensuring proper discontinuity in the soft sense. 

We address a formal inconsistency in the definition of soft covering maps proposed in \cite{saleem2024soft}. We provide a corrected formulation, establish its connection with properly discontinuous soft actions, and obtain a partial converse under a freeness assumption. We also connect the notion of soft covering map with the framework of soft quotient topology developed in this paper. 

Finally, to illustrate a potential use case our framework , we address multi-agent motion planning under parametric environmental uncertainty. By modeling the configuration space as a parameterized soft set, the soft orbit space projection reduces the search space by a factor of $N!$ for $N$ identical agents. Crucially, our soft covering maps guarantee that this compressed space preserves local topological structures, completely avoiding topological distortion and phantom collisions.
\subsection*{Organisation of the paper}
This paper is organised as follows. Section~\ref{sec:prelim} reviews the basic preliminaries on soft sets and soft topological spaces. Section~\ref{sec:soft-quotient-maps} develops the theory of soft quotient maps, including their fundamental properties, slice behaviour, characterisations, and the universal property. Section~\ref{sec:examples} develops key examples, including gluing constructions, soft cones, and soft suspensions and analyse their slice-wise classical counterparts. Section~\ref{sec:soft-actions} studies soft orbit spaces and properly soft discontinuous group actions and their properties. Section~\ref{sec:coversoft} deals with soft covering spaces and their relationship with properly soft discontinuous actions. Finally, Section~\ref{application} applies this theoretical framework to multi-agent motion planning under environmental uncertainty, demonstrating a mathematically justified symmetry reduction.

\section{Preliminaries on soft topological spaces}
\label{sec:prelim}
We recollect the basic notions of soft set theory and soft topology \cite{Molodtsov1999SoftSets,ShabirNaz2011SoftTopo,BahredarKouhestaniPassandidehFundamentalGroup,agman2011SoftT}. Throughout, $X$ denotes a non-empty set (the universe) and $A$ a non-empty set of parameters. We denote the power set of $X$ by $\mathcal{P}(X)$. 
\subsection{Soft sets}
We recall the foundational definitions and standard operations on soft sets that will be used throughout the paper.
\begin{definition}\cite{Molodtsov1999SoftSets}
Let $X$ be a common universe and $A$ a set of parameters. The collection of all soft sets over $X$ with respect to $A$ is $SS(X,A) = \{(F,A)\hspace{1mm}|\hspace{1mm} F:A\to\mathcal{P}(X)\}$. Each pair $(F,A)\in SS(X,A)$ is called a \emph{soft set} over $X$ (with respect to $A$).
\end{definition}

We recall some standard terminology following \cite{maji2003soft,agman2011SoftT,hida2014comparison}:
\begin{enumerate}[noitemsep,topsep=0pt,parsep=0pt]
\item Let $(F,A)$ and $(G,B)$ be soft sets over a universe $X$. We say that $(F,A)$ is a \emph{soft subset} of $(G,B)$, written \((F,A)\sqsubseteq(G,B)\), if $A \subseteq B$ and $F(a) \subseteq G(a) \text{ for all } a\in A$.
\item $(F,A)$ and $(G,A)$ are \emph{soft equal} if each is a soft subset of the other.
\item The \emph{soft complement} of $(F,A)$ is $(F^c,A)$, where $F^c(a)=X\setminus F(a)$ for each $a\in A$.
\item The \emph{null soft set} $(0,A)$ and the \emph{absolute soft set} $(1,A)$ are defined by $0(a)=\emptyset$ and $1(a)=X$ for all $a\in A$.
\item For any subset $S\subseteq X$, we denote by $\widetilde{S}$ the soft set $(F,A)$ defined by $F(a)=S \text{ for every } a\in A$.
\item For a point $x\in X$ and a soft set $(F,A)$ on $X$, we write $x \widetilde\in (F,A)$ if $x \in F(a)$ for every $a \in A$. Note that, for any $x\in X$, $x\centernot{\widetilde{\in}} (F,A)$, if $x\notin F(a)$ for some $a\in A$.
\item Given a family $\{(F_i,A)\}_{i\in I}\subseteq SS(X,A)$, their soft union $\bigsqcup_{i\in I}(F_i,A)$ and soft intersection $\sqcap_{i\in I}(F_i,A)$ are defined by 
\[
\Bigl(\bigsqcup_{i\in I}F_i\Bigr)(a)=\bigcup_{i\in I}F_i(a) \text{ and }
\Bigl(\sqcap_{i\in I}F_i\Bigr)(a)=\bigcap_{i\in I}F_i(a),
\] respectively, for all $a\in A$.
\end{enumerate}

\subsection{Soft topological spaces}
We recall the axioms for soft topological spaces and review the notion of parametric slice topologies.
\begin{definition}\cite{ShabirNaz2011SoftTopo,agman2011SoftT}
A \emph{soft topology} on $X$ (with respect to $A$) is a subfamily $\tau\subseteq SS(X,A)$ such that:
\begin{enumerate}[label=(\roman*),noitemsep,topsep=0pt,parsep=0pt]
\item $(0,A),(1,A)\in\tau$;
\item if $(F_i,A)\in\tau$ for all $i\in I$, then $\bigsqcup_{i\in I}(F_i,A)\in\tau$;
\item if $(F_1,A), (F_{2},A),\dots,(F_n,A)\in\tau$ for some $n\in\mathbb{N}$, then $(F_1,A)\sqcap (F_2,A)\sqcap \cdots \sqcap (F_n,A)\in\tau$.
\end{enumerate}
The triple $(X,\tau,A)$ is called a \emph{soft topological space}, and the members of $\tau$ are called \emph{soft open sets}. A soft set is called \emph{soft closed} if its complement is soft open.
\end{definition}
\begin{itemize}[noitemsep,topsep=0pt,parsep=0pt]
\item[(a)] For a soft space $(X,\tau,A)$ and a fixed parameter $a\in A$, the family $\tau_a=\{F(a) : (F,A)\in\tau\}$ is a (classical) topology on $X$, called the \emph{$a$-slice topology}. We denote the corresponding topological space $(X,\tau_a)$ by $X_a$.
\item[(b)] If the soft topology $\tau$ satisfies 
\[
(F,A)\in \tau_{X}
\quad\iff\quad
F(a)\in \tau_{a} \ \text{for all } a\in A,
\]then we refer to $\tau$ as a \emph{slice-determined soft topology}. Note that a slice–determined soft topology need not be unique. For instance, let $(X,\mathcal{T})$ be a topological space. We may define two slice–determined soft topologies $\tau_{1}$ and $\tau_{2}$ on $X$ as follows: \[
\tau_{1}=\{(F,A) : F(a)\in\mathcal{T}\ \text{for all } a\in A\},
\qquad
\tau_{2}=\{(\lambda_{U},A) : U\in\mathcal{T}\},
\] where $\lambda_{U}(a)=U$ for every $a\in A$. It is immediate that $\tau_{2}\subset \tau_{1}$, and both are slice–determined soft topologies on $X$, with each of their slices coinciding with $\mathcal{T}$.
\item[(c)] For $x\in X$, a soft set $(F,A)$ on $X$ is called a \emph{soft neighborhood} of $x$ if $x\widetilde{\in}(F,A)$ and $(F,A)\in\tau$.
\item[(d)] We say a subset $S$ of $X$ is soft open (resp. soft closed) in $(X,\tau,A)$ if the soft set $\widetilde{S}\in\tau$ (resp. $\notin\tau)$.
\end{itemize}

\subsection{Soft continuity}
We recall the notion of soft continuity from \cite{BahredarKouhestaniPassandidehFundamentalGroup}. For this, we first review soft maps between soft sets.
\begin{definition}\label{softlevel}
Let $SS(X,A)$ and $SS(Y,B)$ be families of soft sets. Let $f:X\to Y$ be a map and $e:A\to B$ a map of parameter sets (called a \emph{parametric map}). Define $\phi_{f,e}:SS(X,A)\rightarrow SS(Y,B)$ by assigning to $(F,A)$ the soft set $(G,B)$ with 
\[
G(b)=
\begin{cases}
\displaystyle\bigcup_{a\in e^{-1}(b)} f(F(a)), & e^{-1}(b)\neq\emptyset,\\[0.3em]
\emptyset, & e^{-1}(b)=\emptyset.
\end{cases}
\]
The \emph{inverse soft image} of a soft set $(G,B)$ under $\phi_{f,e}$ is the soft set $\phi^{-1}_{f,e}(G,B)=(F,A)$ with $F(a)=f^{-1}(G(e(a)))$ for all $a\in A$.
\begin{remark}\cite{BahredarKouhestaniPassandidehFundamentalGroup}
If the map $e$ is a bijection, then the followings hold:
\begin{enumerate}[noitemsep,topsep=0pt,parsep=0pt]
    \item $G(b)=f(F(e^{-1}(b))),\hspace{1mm}\forall b\in B $.
\item If $f$ is a bijection then $\phi_{f,e}=(\phi_{f^{-1},e^{-1}})^{-1}$.
\end{enumerate}
\end{remark}
\end{definition}

\begin{definition}
Let $(X,\tau_X,A)$ and $(Y,\tau_Y,B)$ be soft topological spaces and $e:A\to B$ a parametric map. A map $f:X\to Y$ is called \emph{soft $e$-continuous} if for every soft open set $(G,B)\in\tau_Y$ we have $\phi^{-1}_{f,e}(G,B)\in\tau_X$.
\end{definition}
\begin{definition}
    Let $(X,\tau_X,A)$ and $(Y,\tau_Y,B)$ be soft topological spaces and $e:A\to B$ a bijective parametric map. A bijective map $f:X\to Y$ is called \emph{soft $e$-homeomorphism} if $f$ and $f^{-1}$ are soft $e$-continuous and soft $e^{-1}$-continuous respectively.
\end{definition}

\subsection{Subspaces and products}
In this subsection, we recall the basic constructions of soft subspaces and soft products.
\begin{definition}\cite{ShabirNaz2011SoftTopo}
 \hspace{1mm}Let $(X,\tau_X,A)$ be a soft topological space and $S\subseteq X$ non-empty. The \emph{soft subspace} $(S,\tau_S,A)$ is defined by $\tau_S=\{(F_S,A):(F,A)\in\tau_X\},$ where $F_S(a)=F(a)\cap S$ for each $a\in A$. Equivalently, $\tau_{S}=\tau_{X}\sqcap\widetilde{S}$.
\end{definition}

\begin{definition}\cite{hida2014comparison}
 Let $(X,\tau_X,A)$ and $(Y,\tau_Y,B)$ be soft topological spaces. Given soft sets $(F,A)\in SS(X,A)$ and $(G,B)\in SS(Y,B)$, their \emph{rectangular soft product} is the soft set $(F,A)\boxtimes(G,B)=(H,A\times B)$, where $H(a,b)=F(a)\times G(b)$ for all $(a,b)\in A\times B$. The \emph{soft product} of $(X,\tau_X,A)$ and $(Y,\tau_Y,B)$ is the soft space
\[
(X,\tau_X,A)\times(Y,\tau_Y,B)
:= (X\times Y,\tau_{X\times Y},A\times B),
\]
where $\tau_{X\times Y}$ is the soft topology on $X\times Y$ generated by $\{(F,A)\boxtimes (G,B)\colon (F,A)\in\tau_{X}, (G,B)\in\tau_{Y}\}$.
\end{definition}

\begin{remark}\label{proj}
The canonical projections $\pi_X:X\times Y\to X$ and $\pi_Y:X\times Y\to Y$ are soft $p_A$-continuous and soft $p_B$-continuous, respectively, where $p_A:A\times B\to A,\; p_A(a,b)=a$, and $p_B:A\times B\to B,\; p_B(a,b)=b$.
\end{remark}

\subsection*{Convention} Throughout the rest of the paper we adopt the following convention.
\begin{enumerate}[noitemsep,topsep=0pt,parsep=0pt]
    \item Let $(X,\tau_{X},A)$ and $(Y,\tau_{Y},B)$ be soft topological spaces, let $e\colon A \to B$ be a parametric map, and let $f\colon X \to Y$ be a map of the underlying universes. Whenever no confusion arises, the pair $(f,e)\colon (X,\tau_{X},A) \to (Y,\tau_{Y},B)$ is termed a \emph{soft map}.
    \item We say that the pair $(f,e)$ is soft continuous whenever the map $f$ is soft $e$-continuous. The same convention is adopted for all other ``soft properties" of $f$.
    \item The map $\phi_{f,e}$ introduced in Definition~\ref{softlevel} is referred to as the \emph{induced map between soft sets}.
\end{enumerate}

\section{Soft quotient maps and soft quotient topology}
\label{sec:soft-quotient-maps}
In this section, we introduce soft quotient maps and the related soft quotient topology, examining their fundamental properties. We analyze the relationship between soft quotient maps and their slice-wise counterparts. Additionally, we establish the universal property of soft quotient spaces and use it in subsequent arguments.

\subsection{Soft \texorpdfstring{$e$-quotient}{e-quotient} maps}
Let $(X,\tau_X,A)$ and $(Y,\tau_Y,B)$ be soft topological spaces. We say that a soft map $(q,e):(X,\tau_X,A)\to(Y,\tau_Y,B)$ is \emph{soft surjective} if the induced map $\phi_{f,e}\colon SS(X,A)\rightarrow SS(Y,B)$ is \emph{surjective}.
\begin{remark} 
It is easy to observe that the induced map $\phi_{q,e}$ is surjective if and only if both the underlying map $q \colon X \to Y$ and the parametric map $e \colon A \to B$ are surjective.
\end{remark}
\begin{definition}
We say a soft surjective map $(q,e):(X,\tau_X,A)\to(Y,\tau_Y,B)$ to be a soft $e$-quotient map if, a soft set $(G,B)$ over $Y$ is soft open if and only if its inverse soft image $\phi^{-1}_{q,e}(G,B)$ over $X$ is soft open. An equivalent condition is to require that a soft set $(F,B)$ over $Y$ is soft closed if and only if its inverse soft image $\phi^{-1}_{q,e}(F,B)$ over $X$ is soft closed.
\end{definition}
\begin{Theorem}\label{prop3.3}
Let $(X,\tau_X,A)$ be a soft space and $q:X\to Y$ a surjective map onto a set $Y$ together with a surjective parametric map $e:A\to B$. Then,
\[\tau_{q,e}=\{(G,B)\in SS(Y,B):\phi^{-1}_{q,e}(G,B)\in\tau_X\}\]
is a soft topology on $Y$, and with this topology $(q,e):(X,\tau_X,A)\to(Y,\tau_{q,e},B)$ is a soft quotient map.
\end{Theorem}
\begin{proof}
The verification that $\tau_{q,e}$ satisfies the axioms of a soft topology is routine and follows directly from the corresponding properties of inverse soft images and the fact that $\tau_X$ is a soft topology. By construction, $(G,B)\in\tau_{q,e} \iff \phi^{-1}_{q,e}(G,B)\in\tau_X,$ and hence $q$ is a soft $e$-quotient map.
\end{proof}
\begin{remark}
    The topology $\tau_{q,e}$ is the unique topology on $Y$ with respect to which
    
    $(q,e)\colon (X,\tau_{X},A)\rightarrow(Y,\tau_{q,e},B)$ is a soft quotient map.
\end{remark}

\begin{definition}\label{quotient_soft}
Let $(X,\tau_X,A)$ be a soft topological space, and let $q:X\to Y$ be a surjective map onto a set $Y$. Let $e:A\to B$ be a surjective map onto a parameter set $B$. The soft topology $\tau_{q,e}$ on $Y$, defined in Theorem~\ref{prop3.3}, is referred to as the \emph{soft quotient topology} on $Y$ induced by the map $(q,e)$.
\end{definition}

\subsection{Slice-wise description}
Although it is natural to expect that the restriction of a soft quotient map to each slice should be a classical quotient map, this need not hold in general. The following Example~\ref{softquotintnotslicewise} shows that a soft quotient map may fail to induce classical quotient maps on its slices.
\begin{example}\label{softquotintnotslicewise}
Let $A=\{a,b\}$ and $B=\{c\}$, and let $X=Y=\{0,1\}$. Define $q:X\to Y$ to be the identity map and let $e:A\to B$ be given by $e(a)=e(b)=c$.

Define a soft topology $\tau_X$ on $(X,A)$ by declaring that a soft set $(F,A)$ over $X$ is soft open if and only if $F(b)=X$ together with the null soft set $(0,A)$. Its slice topologies are $(\tau_X)_a=\mathcal{P}(X)$ and $(\tau_X)_b=\{\emptyset,X\}.$ 

Equip $(Y,B)$ with the soft topology $\tau_Y=\{(0,B),\ (Y,B)\}$, so that $(\tau_Y)_c=\{\emptyset,Y\}$ is the indiscrete topology. A direct verification shows that $q:X\to Y$ is a soft $e$-quotient map. However, the slice map $q_a:X_a\to Y_c$ is the identity from a discrete space to an indiscrete space and hence fails to be a classical quotient map.
\end{example} 
In contrast, under suitable conditions on the soft topology of the domain, a soft quotient map induces classical quotient maps on the slices. This is established in the following theorem.
\begin{Theorem}\label{prop:slice-quotient-general}
Let $(X,\tau_X,A)$ and $(Y,\tau_Y,B)$ be soft topological spaces, and let $(q,e):(X,\tau_X,A)\to (Y,\tau_Y,B)$ be a soft quotient map. If the soft topology $\tau_{X}$ is a slice-determined soft topology then, for each parameter $a\in A$, the slice map $q_a : X_a \rightarrow Y_{e(a)}$ with $q_a(x)=q(x),$ is a quotient map with respect to the slice topologies $(\tau_{X})_{a}$ on $X_{a}$ and $(\tau_{Y})_{e(a)}$ on $Y_{e(a)}$.
\end{Theorem}
\begin{proof}
Fix $a\in A$. Let $U\subseteq Y_{e(a)}$. If $U\in(\tau_Y)_{e(a)}$, then by definition there exists a soft open set $(G,B)$ over $Y$ with $G(e(a))=U$. Since $q$ is a soft $e$-quotient map, it follows that $\phi^{-1}_{q,e}(G,B)\in\tau_X$. This yields, $q_a^{-1}(U)=\phi^{-1}_{q,e}(G,B)(a)\in(\tau_X)_a.$ Conversely, suppose $q_a^{-1}(U)\in(\tau_X)_a$. Define a soft set $(G,B)$ on $Y$ by setting $G(e(a))=U$ and $G(b)=\emptyset$ for $b\neq e(a)$. Then $\phi^{-1}_{q,e}(G,B)(a)=q_a^{-1}(U)$ and all other slices are empty. Hence $\phi^{-1}_{q,e}(G,B)\in\tau_X$. Since $q$ is a soft $e$-quotient map, it follows that $(G,B)\in\tau_Y$, and hence $U\in(\tau_Y)_{e(a)}$. Thus, $U\in(\tau_Y)_{e(a)}$ if and only if $q_a^{-1}(U)\in(\tau_X)_a$, proving that $q_a:X_a\to Y_{e(a)}$ is a quotient map.
\end{proof}
The converse of Theorem~\ref{prop:slice-quotient-general} does not hold in general. Example~\ref{exmple324} demonstrates that the requirement for all slice maps to be classical quotient maps is insufficient to guarantee that the corresponding soft map is a soft quotient map.
\begin{example}\label{exmple324}
Let $E=\{e_1,e_2\}$ and $X=Y=\{0,1\}$. Let $q:X\to Y$ be the identity map with $e=\id_E$. Equip $X$ with the discrete soft topology $\tau^{discrete}_X=SS(X,E)$ and equip $Y$ with the diagonal soft topology $\tau^{\Delta}_Y=\{(G,E): G(e_1)=G(e_2)\}.$ Then each slice $X_{e_i}$ and $Y_{e_i}$ is discrete. Consequently, for each $e_i\in E$, the slice map $q_{e_i}:X_{e_i}\to Y_{e_i}$ is a classical quotient map. However, $q$ is not a soft $e$-quotient map. Indeed, if $(H,E)$ is defined by $H(e_1)=\{0\}$ and $H(e_2)=\emptyset$, then $(H,E)\notin\tau^{\Delta}_Y$, while $\phi^{-1}_{q,e}(H,E)=(H,E)\in\tau^{discrete}_X$.
\end{example}
But when both the domain and co-domain carry slice–determined soft topologies, then converse of Theorem~\ref{prop:slice-quotient-general} also holds. In fact, we obtain the following theorem. Its proof is completely analogous to that of Theorem~\ref{prop:slice-quotient-general} and is therefore omitted.
\begin{theorem}\label{sliceimplysoftquotirnt}
    Let $(X,\tau_X,A)$ and $(Y,\tau_Y,B)$ be soft topological spaces, and let $(q,e):(X,\tau_X,A)\to (Y,\tau_Y,B)$ be a soft surjective map. If the soft topologies $\tau_{X}$ and $\tau_{Y}$ are both slice-determined soft topologies then $q$ is soft $e$-quotient map if and only if for each parameter $a\in A$, the slice map $q_a : X_a \rightarrow Y_{e(a)}$ is a quotient map with respect to the slice topologies.
\end{theorem}
\begin{corollary}\label{cor:slice-quotient-topology}
Let $(X,\tau_{X},A)$ be a soft topological space, and let $(q,e):(X,\tau_X,A)\to (Y,\tau_{q,e},B)$ be a soft quotient map, where $\tau_{q,e}$ denotes the soft quotient topology induced by $(q,e)$ in the sense of Definition~\ref{quotient_soft}. If $\tau_X$ is a slice–determined soft topology, then for each $a\in A$ the slice topology on $Y_{e(a)}$ induced by the soft quotient topology $\tau_{q,e}$ coincides with the classical quotient topology determined by the slice map $q_a:X_a\to Y_{e(a)}$.
\end{corollary}
\begin{remark}
Without the assumption that $\tau_{X}$ is a slice-determined soft topology, the slice topology on $Y_{e(a)}$ induced from the soft quotient topology $\tau_{q,e}$ is in general, \emph{coarser} than the classical quotient topology on $Y_{e(a)}$ determined by the slice map $q_{a}$.
\end{remark}

We now consider quotient maps arising from equivalence relations on the underlying set of a soft space. Let $(X,\tau_X,E)$ be a soft space and $\sim$ an equivalence relation on $X$. Denote by $X/{\sim}$ the set of equivalence classes and by $\pi:X\to X/{\sim}$ the canonical projection.
\begin{definition}\label{soft!quotient}
The \emph{soft quotient space} of $(X,\tau_X,E)$ by $\sim$ is the soft space $(X/{\sim},\tau_\pi,E)$ where $\tau_\pi$ is the soft quotient topology induced by $(\pi,\id_E)$, i.e. $\tau_\pi=\{(G,E):\phi^{-1}_{\pi,\id_E}(G,E)\in\tau_X\}.$
\end{definition}
As a direct application of Theorem~\ref{prop:slice-quotient-general} with \((q,e) = (\pi,\id_E)\), we obtain the following result.
\begin{Theorem}\label{slicequotient}
If $\tau_{X}$ is a slice-determined soft topology then, for each parameter $e\in E$, the slice topology on $(X/{\sim})_e$ is the classical quotient topology induced by the classical projection $\pi_e:X_e\to X_e/{\sim}$.
\end{Theorem}

\subsection{Compositions and restrictions}
We next describe the behaviour of soft quotient maps under composition and restriction.
\begin{Theorem}
Let $(q_1,e_1):(X,\tau_X,A)\to(Y,\tau_Y,B)$ and $(q_2,e_2):(Y,\tau_Y,B)\to(Z,\tau_Z,C)$ be soft quotient maps, respectively. Then the composite $(q_2\circ q_1,e_2\circ e_1):(X,\tau_X,A)\rightarrow(Z,\tau_Z,C)$ is a soft quotient map.
\end{Theorem}
\begin{proof}
The composite $(q_2\circ q_1,\,e_2\circ e_1)$ is clearly soft surjective. Soft $(e_2\circ e_1)$-continuity of $(q_{2}\circ q_{1})$ follows from the identity $\phi^{-1}_{q_2\circ q_1,\,e_2\circ e_1} = \phi^{-1}_{q_1,e_1}\circ \phi^{-1}_{q_2,e_2}.$ Let $(K,C)$ be a soft set on $Z$. Using the above formula and the fact that $(q_1,e_1)$ and $(q_2,e_2)$ are soft quotient maps, we obtain
\[
(K,C)\in\tau_Z \;\iff\; \phi^{-1}_{q_2,e_2}(K,C)\in\tau_Y \;\iff\; \phi^{-1}_{q_2\circ q_1,\,e_2\circ e_1}(K,C)\in\tau_X.
\]
Hence $q_2\circ q_1$ is a soft $e_2\circ e_1$-quotient map.
\end{proof}
In general, the restriction of a soft quotient map to an arbitrary subset need not be a soft quotient map; see \cite[Example~1]{MunkresTopology} for the classical analogue. Theorem~\ref{qres} gives a sufficient condition under which such a restriction remains a soft quotient map. Before stating the theorem, we recall the following definition from \cite{MunkresTopology}.
\begin{definition}
Let $q:X\to Y$ be map between sets. A subset $S\subseteq X$ is called \emph{$q$-saturated} if $S=q^{-1}(q(S))$.
\end{definition}
\begin{remark}
Let $A$ and $B$ are set of set parameters on universes $X$ and $Y$ respectively. Let $(q,e)\colon (X,A)\rightarrow (Y,B)$ be a map between them. If $S$ is a saturated subset of $X$, then one can check that the soft inverse image of $\widetilde{q(S)}$, i.e., $\phi^{-1}_{q,e}(\widetilde{q(S))}$ is nothing but $\widetilde{S}$.
\end{remark}
Now we have the following theorem.
\begin{Theorem}\label{qres}
Let $(q,e):(X,\tau_X,A)\to(Y,\tau_Y,B)$ be a soft quotient map and let $S\subseteq X$ be a $q$-saturated subset. Equip $S$ with the subspace soft topology and $q(S)\subseteq Y$ with the subspace soft topology inherited from $(Y,\tau_Y,B)$. If $S$ is either soft open or soft closed in $(X,\tau_X,A)$, Then the restricted map $(q|_S,e):(S,\tau_S,A)\rightarrow(q(S),\tau_{q(S)},B)$ is a soft quotient map.
\end{Theorem}

\begin{proof}
Assume first that $S$ is soft open in $(X,\tau_X,A)$. Let $(V,B)$ be a soft set over $q(S)$. Since $S$ is $q$-saturated, it follows that $\phi^{-1}_{q|_S,e}(V,B)=\phi^{-1}_{q,e}(V,B)$. If $(V,B)$ is soft open over $q(S)$, then by definition of the soft subspace topology, there exists a soft open set $(G,B)$ over $Y$ such that $(V,B)=(G,B)\sqcap\widetilde{q(S)}$. Hence $\phi^{-1}_{q|_S,e}(V,B) =\phi^{-1}_{q,e}(G,B)\sqcap\widetilde{S}.$ Since $q$ is soft $e$-continuous and $S$ is soft open, it follows that the right-hand side is soft open in $S$. Conversely, if $\phi^{-1}_{q|_S,e}(V,B)$ is soft open over $S$, then (as $S$ is soft open in $X$) it is soft open over $X$. Since $(q,e)$ is a soft $e$-quotient map, this implies $(V,B)$ is soft open in $Y$, and hence soft open in $q(S)$. The case where $S$ is soft closed in $(X,\tau_X,A)$ follows by an analogous argument.
\end{proof}

\subsection{Characterisations via openness and closedness}
Recall that in classical topology, a surjective continuous map that is either open or closed is a quotient map \cite[Theorem 4.3]{MunkresTopology}. This result extends to the soft topological context. Before proceeding, we recall the definitions of soft open and soft closed maps.
\begin{definition}\label{opensoft} 
Let $(f,e):(X,\tau_X,A)\to(Y,\tau_Y,B)$ be a soft map. The map $f$ is called a \emph{soft $e$-open map} (respectively, a \emph{soft $e$-closed map}) if for every soft open (respectively, soft closed) set $(F,A)$ over $X$, the image soft set $\phi_{f,e}(F,A)$ is a soft open (respectively, soft closed) set over $Y$.
\end{definition}
\begin{remark}\label{softopenusualopen}
   It is easy to observe that if $(f,e):(X,\tau_X,A)\to (Y,\tau_Y,B)$ is soft open map (respectively, soft closed map), then for every parameter $a\in A$ the slice map $f_a \;:=\; f|_{X_a} : X_a \rightarrow Y_{e(a)}$ is an open (respectively, closed) map in the classical sense.
\end{remark}
\begin{theorem}\label{thm:open-closed-quotient}
Let $(q,e):(X,\tau_X,A)\to(Y,\tau_Y,B)$ be a surjective soft continuous map. If $q$ is soft $e$-open or soft $e$-closed, then $q$ is a soft $e$-quotient map.
\end{theorem}
\begin{proof}
We give the argument in the soft $e$-open case, the soft $e$-closed case is entirely analogous. Since $q$ is soft $e$-continuous, so $(G,B)\in\tau_Y \implies \phi^{-1}_{q,e}(G,B)\in\tau_X.$ Conversely, suppose $\phi^{-1}_{q,e}(G,B)=(F,A)\in\tau_X$ and set $(H,B)=\phi_{q,e}(F,A)$. By the soft $e$-open property of $q$, we have $(H,B)\in\tau_Y$. For $b\in B$, 
\[
H(b)=\bigcup_{a\in e^{-1}(b)} q(F(a)) =\bigcup_{a\in e^{-1}(b)} q\bigl(q^{-1}(G(e(a)))\bigr).
\]
Since $q$ is surjective and $e(a)=b$ for $a\in e^{-1}(b)$, it follows that $H(b)=G(b)$ for all $b\in B$. Thus $(H,B)=(G,B)$, and hence $(G,B)\in\tau_Y$. Therefore $(G,B)\in\tau_Y \iff \phi^{-1}_{q,e}(G,B)\in\tau_X,$ so $q$ is a soft $e$-quotient map.
\end{proof}
\begin{corollary}
    For two soft spaces $(X,\tau_{X},A)$ and $(Y,\tau_{Y},B)$, the canonical projections $\pi_{X}\colon X\times Y \to X$ and $\pi_{Y}\colon X\times Y \to Y$, defined in Remark~\ref{proj}, are soft $p_{A}$-open and soft $p_{B}$-open maps, respectively. Consequently, $\pi_{X}$ is a soft $p_{A}$-quotient map and $\pi_{Y}$-quotient map.
\end{corollary}
The converse of Theorem~\ref{thm:open-closed-quotient} does not hold in general, as shown in the following example.
\begin{example}
Let $E=\{e_1,e_2\}$, $X=\{1,2,3\}$, $Y=\{\alpha,\beta\}$, and define a surjection $q:X\to Y$ by $q(1)=q(2)=\alpha$ and $q(3)=\beta$. Define soft topologies $\tau_{X}$ and $\tau_{Y}$ on $X$ and $Y$ respectively by specifying, the slice topologies: 
\[
\tau_X^{(1)}=\{\emptyset,\{1\},\{2,3\},X\},\quad \tau_X^{(2)}=\mathcal{P}(X),\qquad \tau_Y^{(1)}=\{\emptyset,Y\},\quad \tau_Y^{(2)}=\mathcal{P}(Y).
\]
It is straightforward to verify that $(q,\id_E):(X,\tau_X,E)\rightarrow(Y,\tau_Y,E)$ is a soft quotient map. But, $q$ is neither soft $\id_E$-open nor soft $\id_E$-closed map. Indeed, if $(F,E)$ is defined by $F(e_1)=\{1\}$ and $F(e_2)=\emptyset$, then $(F,E)\in\tau_X$ but $\phi_{q,\id_E}(F,E)(e_1)=\{\alpha\}\notin\tau_Y^{(1)},$ so $q$ is not soft $\id_E$-open map. Taking complements shows that it is not soft $\id_E$-closed map too.
\end{example}
Next, we recall the notions of soft compact spaces and soft Hausdorff spaces.
\begin{definition}{\cite{ShabirNaz2011SoftTopo}}
A soft topological space $(Y,\tau_Y,E)$ is said to be \emph{Soft Hausdorff} if for any two distinct points $x,y\in Y$ there exist soft open sets $(G_{1},E),(G_{2},E)\in\tau_Y$ such that 
\[
x\widetilde{\in} (G_{1},E), \qquad y\widetilde{\in} (G_{2},E), \qquad (G_{1},E)\sqcap (G_{2},E)=(0,E).
\]
\end{definition}
There are two inequivalent notions of soft compact spaces \cite{hida2014comparison}.
\begin{definition}\cite{hida2014comparison}
\par\textbf{(First definition)} Let $(X,\tau_X,E)$ be a soft topological space, and let $\mathcal{C}\subseteq \tau_X$ be a family of soft open sets over $X$. We say that $\mathcal{C}$ is a \emph{Soft Covering} (SCV1) of $X$ if for every $x\in X$, there exists $(F,E)\in \mathcal{C}$ such that $x\widetilde{\in} (F,E)$. The soft topological space $(X,\tau_X,E)$ is called \emph{Soft Compact} (SCPT1) if every SCV1 of $X$ admits a finite subfamily which is again SCV1 of $X$.
\par\textbf{(Second definition)} Let $(X,\tau_X,E)$ be a soft topological space, and let $\mathcal{C}\subseteq \tau_X$ be a family of soft open sets over $X$. We say that $\mathcal{C}$ is a \emph{Soft Covering} (SCV2) of $X$ if 
\[
\bigsqcup_{(F,E)\in\mathcal{C} }(F,E)=(1,E).
\]
The soft topological space $(X,\tau_X,E)$ is called \emph{Soft Compact} (SCPT2) if every SCV2 of $X$ admits a finite subfamily which is again SCV2 of $X$.
\end{definition}

In classical topology, a continuous surjection from a compact space onto a Hausdorff space is a quotient map \cite[Corollary~4.4]{armstrong2013basic}. This conclusion does not, in general, extend to the soft topological setting, as shown by the following example.
\begin{example}
Let $X=Y=\{0,1\}$, $E=\{e_1,e_2\}$. Let the parametric map be $e=\id_E$ and let $q:X\to Y$ be the identity map. Equip $X$ with the discrete soft topology $\tau^{discrete}_X$ and $Y$ with the diagonal soft topology $\tau^{\Delta}_Y$. Clearly $(X,\tau^{discrete}_X,E)$ is soft compact in both senses (SCPT1 and SCPT2), while $(Y,\tau^{\Delta}_Y,E)$ is soft Hausdorff. Define $(H,E)$ on $Y$ by $H(e_1)=\{0\}$ and $H(e_2)=\emptyset$. Then $(H,E)\notin\tau^{\Delta}_Y$, but $\phi^{-1}_{q,e}(H,E)=(H,E)\in\tau^{discrete}_X$. Hence $q$ is not a soft $e$-quotient map.
\end{example}
Note that in the above example the soft topology on the domain is a slice-determined soft topology, whereas the soft topology on the codomain is not. The following theorem shows that, if both the domain and codomain carry slice-determined soft topologies, then a soft analogue of \cite[Corollary~4.4]{armstrong2013basic} holds.
\begin{theorem}
Let $(q,e)\colon (X,\tau_X,A)\to (Y,\tau_Y,B)$ be a  soft surjective map. If $(X,\tau_X,A)$ is soft compact in either sense \emph{(SCPT1 or SCPT2)}, $(Y,\tau_Y,B)$ is soft Hausdorff and both $\tau_X$ and $\tau_Y$ are slice-determined soft topologies, then $q$ is a soft $e$-quotient map.
\end{theorem}
\begin{proof}
We first show that each slice $X_{a}$ is compact. Fix $a\in A$ and let $\{U_\alpha\}_{\alpha\in\Lambda}$ be an open cover of the slice space $X_a$. For each $\alpha\in\Lambda$, define a soft set $(F_\alpha,A)$ over $X$ by $F_\alpha(a)=U_\alpha \text{ and } F_\alpha(a')=X \text{ for all } a'\neq a,$ which is soft open. Then $\{(F_\alpha,A)\}_{\alpha\in\Lambda}$ is a soft open cover in the sense of \emph{SCV1} (resp. \emph{SCV2}) of $(X,\tau_X,A)$. By soft compactness of $X$ in the sense of \emph{SCPT1} (resp. \emph{SCPT2}), there exists a finite subset $\Lambda_0\subseteq\Lambda$ such that $\{U_\alpha\}_{\alpha\in\Lambda_0}$ covers $X_a$. Hence each slice $X_a$ is compact. Since $(Y,\tau_Y,B)$ is soft Hausdorff, each slice $Y_b$ is Hausdorff. Therefore, for every $a\in A$, the slice map $q_a \colon X_a \rightarrow Y_{e(a)}$ is a surjective continuous map from a compact space to a Hausdorff space, and hence is a quotient map \cite[Corollary~4.4]{armstrong2013basic}. As both $\tau_X$ and $\tau_Y$ are slice-determined soft topologies, Theorem~\ref{sliceimplysoftquotirnt} applies and yields that $q$ is a soft $e$-quotient map.
\end{proof}

\subsection{Universal property}
Now we describe the universal property of the soft quotient space.
\begin{theorem}
\label{thm:soft-universal}
Let $(X,\tau_X,A)$, $(Y,\tau_Y,B)$ and $(Z,\tau_Z,C)$ be soft spaces. Let $q:X\to Y$ be a soft $e$-quotient map. Suppose $(g,d):(X,\tau_X,A)\to(Z,\tau_Z,C)$ be a soft map such that:
\begin{enumerate}[noitemsep,topsep=0pt,parsep=0pt]
\item $g$ is constant on the fibres of $q$, i.e.\ $q(x)=q(x')$ implies $g(x)=g(x')$ for all $x,x'\in X$;
\item The map $d:A\to C$ is constant on each fibre of $e$. Equivalently, since $e$ is surjective, there exists a (necessarily unique) parametric map $h:B\to C$ such that $d = h\circ e$.
\end{enumerate}
Then,
\begin{enumerate}[label=(\alph*),noitemsep,topsep=0pt,parsep=0pt]
\item There exists a unique map $f:Y\to Z$ such that $f\circ q = g$;
\begin{center}
 \[
\begin{tikzcd}[column sep=large, row sep=large]
(X,\tau_{X},A) \ar[d, swap, "(q\text{,}e)"] \ar[rd, "(g\text{,}d)"]\\
(Y,\tau_{Y},B) \ar[dashed, r,  "(f\text{,}h)"'] & {(Z,\tau_{Z},C).}
\end{tikzcd}
\]
\end{center}
\item $f$ is soft $h$-continuous if and only if $g$ is soft $d$-continuous;
\item $f$ is a soft $h$-quotient map if and only if $g$ is a soft $d$-quotient map.
\end{enumerate}
\end{theorem}

\begin{proof}
(a) Since $g$ is constant on the fibres of $q$, for each $y\in Y$ the set $q^{-1}(y)$ is mapped by $g$ to a single point of $Z$. Define $f(y):=g(x)$ for any $x\in q^{-1}(y)$. This is well defined by the hypothesis. Uniqueness of $f$ follows from the surjectivity of $q$.

\par(b) If $f$ is soft $h$-continuous, then $(g,d)=(f,h)\circ(q,e)$ is soft continuous, since the composition of soft continuous maps is again soft continuous. Conversely, assume $g$ is soft $d$-continuous and let $(H,C)\in\tau_Z$. Using $g=f\circ q$ and $d=h\circ e$, we have $\phi^{-1}_{g,d}(H,C) = \phi^{-1}_{q,e}\bigl(\phi^{-1}_{f,h}(H,C)\bigr).$ Since $\phi^{-1}_{g,d}(H,C)\in\tau_X$ and $(q,e)$ is a soft $e$-quotient map, it follows that $\phi^{-1}_{f,h}(H,C)\in\tau_Y$. Hence $f$ is soft $h$-continuous.
\par(c) Suppose first that $g$ is a soft $d$-quotient map. Then for any soft set $(H,C)$ on $Z$, $(H,C)\in\tau_Z \iff \phi^{-1}_{g,d}(H,C)\in\tau_X.$ Using the identity $\phi^{-1}_{g,d} = \phi^{-1}_{q,e}\circ\phi^{-1}_{f,h}$ and the fact that $q$ is a soft $e$-quotient map, we obtain $\phi^{-1}_{g,d}(H,C)\in\tau_X \iff \phi^{-1}_{f,h}(H,C)\in\tau_Y.$ Hence $f$ is a soft $h$-quotient map. Conversely, if $f$ is a soft $h$-quotient map, the same argument applied in reverse, shows that $g$ is a soft $d$-quotient map.
\end{proof}
Recall that any surjective map $g:X\to Z$ induces an equivalence relation on $X$ by $x\sim x'$ if and only if $g(x)=g(x')$.
\begin{corollary}
Let $(g,d):(X,\tau_X,A)\to(Z,\tau_Z,C)$ be a soft continuous map such that $g:X\to Z$ is surjective, and let $\sim$ be the equivalence relation on $X$ induced by $g$. Let $q:X\to X/{\sim}$ denote the canonical projection, and equip $X/{\sim}$ with the soft $\id_A$-quotient topology $\tau_{X/{\sim}}$ induced by $(q,\id_A)$. Then,
\begin{enumerate}[label=(\alph*),noitemsep,topsep=0pt,parsep=0pt]
\item there exists a unique bijective map $f:X/{\sim}\to Z$ such that $f\circ q=g$;
\item if $d\colon A\to C$ is bijective, then $(f,d):(X/{\sim},\tau_{X/{\sim}},A)\to(Z,\tau_Z,C)$ is a soft homeomorphism if and only if $(g,d)$ is a soft quotient map.
\end{enumerate}
\end{corollary}
\begin{proof} The hypotheses of Theorem~\ref{thm:soft-universal} are satisfied. Therefore, there exists a unique map $f:X/\!\sim\to Z$ such that $f\circ q=g$. Explicitly, $f([x])=g(x)$ for each equivalence class $[x]\in X/\!\sim$. Assertion (a) follows immediately from the definition of $\sim$ and the surjectivity of $g$. Assertion~(b) follows directly from Theorem~\ref{thm:soft-universal}(b),(c), which imply that $(f,d)$ is soft continuous (respectively, a soft $d$-quotient map) if and only if $(g,d)$ has the corresponding property. Being $f$ bijective, $(f,d)$ is a soft homeomorphism if and only if $(g,d)$ is a soft quotient map.
\end{proof}

\section{Examples and constructions in soft quotient topology}
\label{sec:examples}
We now present several examples illustrating soft quotient spaces. Throughout this section, a fixed parameter set $E$ is used and for convenience, all parametric maps are assumed to be the identity unless stated otherwise. Let $I=[0,1]$ be equipped with its usual topology $\mathcal{T}_{\mathrm{std}}$. Define a soft topology $\tau_I$ on $(I,E)$ by
\[
\tau_I=\{(F,E): F(e)\in \mathcal{T}_{\mathrm{std}} \text{ for all } e\in E\}.\]
$\tau_I$ is called the \emph{standard soft topology} on $I$ and the soft space $(I,\tau_I,E)$ is referred to as the \emph{soft interval}. When no ambiguity arises, we denote this space simply by $I$. Note that $\tau_I$ is a slice-determined soft topology and that each slice $(\tau_I)_e$ is $\mathcal{T}_{\mathrm{std}}$.

\subsection{The soft circle as a quotient of the soft interval}
We first specify a soft topology on the unit circle. Let $\mathbb{S}^{1}=\{(x,y)\in\mathbb{R}^{2} : x^{2}+y^{2}=1\}$ endowed with its usual subspace topology. The \emph{standard soft topology} on $\mathbb{S}^1$ is defined by 
\[
\tau_{\mathbb{S}^{1}} = \{(F,E): F(e)\ \text{is open in}\ \mathbb{S}^{1}\ \text{for all}\ e\in E\}.\] The resulting soft space $(\mathbb{S}^{1},\tau_{\mathbb{S}^{1}},E)$ is referred to as the \emph{standard soft circle}. Note that $\tau_{\mathbb{S}^{1}}$ is also a slice-determined soft topology.
Consider the soft interval $(I,\tau_{I},E)$ introduced earlier. Define an equivalence relation $\sim$ on $I=[0,1]$ by identifying the two endpoints:
\[
x\sim y \iff x=y \ \text{ or }\ \{x,y\}=\{0,1\}.\] Let $\pi:I\to I/\!\sim$ denote the canonical projection onto the quotient set. We equip $(I/\!\sim,E)$ with the soft quotient topology $\tau_{\pi}$ induced by the soft map $(\pi,\id_{E})$ as described in Definition~\ref{soft!quotient}.
\begin{Theorem}
\label{prop:soft-circle}
The soft quotient space $(I/\!\sim,\tau_{\pi},E)$ is soft homeomorphic to the soft circle $(\mathbb{S}^{1},\tau_{\mathbb{S}^{1}},E)$.
\end{Theorem}
\begin{proof}
Define $(h,Id_{E})\colon (I,\tau_{I},E)\rightarrow (\mathbb{S}^{1},\tau_{\mathbb{S}^{1}},E)$ by $t\mapsto (\cos 2\pi t,\;\sin 2\pi t)$, which is quotient map in each slice. Hence by Theorem~\ref{sliceimplysoftquotirnt}, $h$ is soft $\id_{E}$-quotient map. So, by universal property the induced map $\widetilde{h} : (I/\!\sim,\tau_{\pi},E) \rightarrow (\mathbb{S}^{1},\tau_{\mathbb{S}^{1}},E)$ defined by $\widetilde{h}([t]):=(\cos 2\pi t,\;\sin 2\pi t)$ is a soft $\id_{E}$-homeomorphism.
\end{proof}
 \begin{remark}
Unlike the uniform circle slices obtained in Theorem~\ref{prop:soft-circle}, Example~\ref{ex:parameter-dependent-homotopy} demonstrates that soft quotient topologies can yield distinct slice-wise homotopy types using the same equivalence relation.
\end{remark}
\begin{example}
\label{ex:parameter-dependent-homotopy}
Let $X=[0,1]$ and $E=\{a,b\}$. We define a soft topology $\tau_X$ on $X$ by specifying the slice topologies: $(\tau_X)_a = \mathcal{T}_{\mathrm{std}}, (\tau_X)_b = \{\emptyset, X\}.$ Let $\sim$ be the equivalence relation on $X$ identifying the endpoints $0 \sim 1$. Let $Y = X/\!\sim$ be the quotient set and $q: X \to Y$ the canonical projection. We equip $Y$ with the soft quotient topology $\tau_{q, \id_E}$. By Corollary~\ref{cor:slice-quotient-topology}, the topology on each slice $Y_e$ coincides with the classical quotient topology induced by the slice map $q_e: X_e \to Y_e$, for $e\in E$.
\begin{itemize}[noitemsep,topsep=0pt,parsep=0pt]
    \item At $e=a$: The map $q_a$ identifies the endpoints of the Euclidean interval, so $Y_a$ is homeomorphic to the circle $S^1$.
    \item At $e=b$:  Since $X_b$ has the indiscrete topology, the quotient topology on $Y_b$ is also indiscrete, and hence $Y_b$ is contractible.
\end{itemize}
Thus, the resulting soft quotient space has slices of distinct homotopy types, $Y_a \simeq S^1,  Y_b \simeq *,$ despite the underlying quotient map $q$ being independent of the parameter.
\end{example}

\subsection{Collapsing a soft subspace to a point}
Let $(X,\tau_X,E)$ be a soft space and $A\subseteq X$ a non-empty subset. Define an equivalence relation on $X$ by $x\sim y \iff \text{$x=y$ or $x,y\in A$}.$ Let $X/A=X/{\sim}$ and endow it with the soft quotient topology induced by $(\pi,\id_E)$, where $\pi:X\to X/A$ is the canonical projection. The resulting soft space $(X/A,\tau_\pi,E)$ is called the soft quotient space obtained from $(X,\tau_X,E)$ by \emph{collapsing $A$ to a single point}.
\begin{theorem}
If $\tau_X$ is a slice-determined soft topology, then for each parameter $e\in E$, the slice $(X/A)_e$ of the soft quotient space $(X/A,\tau_\pi,E)$ coincides with the classical quotient space of $X_e$ obtained by collapsing the subset $A\subseteq X_e$ to a single point.
\end{theorem}
\begin{proof}
The proof follows from Theorem~\ref{prop:slice-quotient-general}.
\end{proof}
This construction appears later in the definitions of cones and suspensions in the soft setting.

\subsection{Gluing Construction}
We now describe the gluing two soft spaces along soft subspaces. Let $(X,\tau_X,A)$ and $(Y,\tau_Y,B)$ be soft spaces. Suppose $A_X \subseteq X$ is a soft subspace with parameters $A$ and $A_Y \subseteq Y$ is a soft subspace with parameters $B$ and $(\theta,h):(A_X,\tau_{A_X},A)\to(A_Y,\tau_{A_Y},B)$ is soft continuous. Form the disjoint union of sets $X \sqcup Y$. Its parameter set is taken to be $A \sqcup B$, and a soft set $(F\sqcup G,A\sqcup B)$ on $X\sqcup Y$ is defined by: 
\[
(F\sqcup G)(c) = 
\begin{cases}
F(c), & if\quad c\in A;\\
G(c), & if\quad c\in B;
\end{cases}
\]
where $(F,A)\in \tau_X$ and $(G,B)\in \tau_Y$. The coproduct soft topology on $X \sqcup Y$ is then given by 
\[
\tau_{X\sqcup Y} = \bigl\{\, (F \sqcup G, A \sqcup B) \; : \; (F,A)\in\tau_{X} \ \text{and} \ (G,B)\in\tau_{Y} \,\bigr\}.
\]
Define an equivalence relation $\sim$ on $X\sqcup Y$ by declaring: $x \sim y \iff \text{$x\in A_X$, $y=\theta(x)\in A_Y$}.$ The canonical projection $\pi : X\sqcup Y \rightarrow (X\sqcup Y)/{\sim}$ is accompanied by the parametric map $\id_{A \sqcup B}$.
\begin{definition}
The \emph{soft gluing space} of $(X,\tau_X,A)$ and $(Y,\tau_Y,B)$ along $(A_X,A)$ and $(A_Y,B)$ via $(\theta,h)$ is the soft space $X\cup_{\theta,h} Y := \bigl( (X\sqcup Y)/{\sim},\; \tau_{\pi},\; A\sqcup B \bigr)$ where $\tau_\pi$ is the soft $\id_{A\sqcup B}$-quotient topology induced by $(\pi,\id_{A\sqcup B})$.
\end{definition}

\begin{Theorem}
Let $(Z,\tau_Z,C)$ be a soft space. Let $(f,u):(X,\tau_X,A)\to(Z,\tau_Z,C)$ and $(g,v):(Y,\tau_Y,B)\to(Z,\tau_Z,C)$ be soft maps such that 
\[
f|_{A_X} = g\circ \theta, \qquad\text{and}\qquad v\circ h = u \quad\text{on }A.
\]
Then there exists a unique soft $u\sqcup v$-continuous map  $w : X\cup_{\theta,h}Y \rightarrow Z$ such that $w\circ\iota_X = f, w\circ \iota_Y = g,$ where $\iota_X,\iota_Y$ are the canonical inclusions into $X\sqcup Y$.
\end{Theorem}
\begin{proof}
Define $f\sqcup g : X\sqcup Y\to Z$ by restriction. The compatibility condition $f|_{A_X}=g\circ\theta$ tells that $f\sqcup g$ is constant on $\sim$-equivalence classes. The conclusion now follows from the universal property.
\end{proof}

\subsection{Soft cone and soft suspension}
\label{sec:cones-suspensions}
Using soft quotient spaces, we define soft analogues of the cone and suspension with a fixed parameter set $E$ and identity parametric maps. Let $(X, \tau_X, E)$ be a soft space and $I$ be the soft interval.
\begin{definition} \label{def:soft-cone}
The \emph{soft cone} on $X$ is the soft quotient space 
\[
\mathrm{Cone}(X) := ((X \times I) / (X \times \{1\}), \tau_\pi, E)\]obtained by collapsing the subspace $X \times \{1\}$ to a point.
\end{definition}
The following theorem establishes that, for slice-determined soft topologies, the parametric slices of the soft cone coincide with classical cones of the underlying space.
\begin{Theorem}\label{cone:slice}
For each $e\in E$, the slice $\mathrm{Cone}(X)_e$ is homeomorphic to the classical cone on the slice $X_e$, if $\tau_{X}$ is a slice-determined soft topology.
\end{Theorem}
\begin{proof}
The result follows from Theorem~\ref{prop:slice-quotient-general}.
\end{proof}
Similarly, we define the soft suspension. Define an equivalence relation $\sim$ on $X\times I$ by 
\[
(x,t)\sim(x',t') \quad\iff\quad \bigl(t=t'=0\bigr)\ \text{or}\ \bigl(t=t'=1\bigr)\ \text{or}\ (x,t)=(x',t').
\]
Let $\Sigma X=(X\times I)/{\sim}$ with projection $\pi:X\times I\to\Sigma X$ and soft quotient topology induced by $(\pi,\id_E)$.
\begin{definition}
The \emph{soft suspension} of $X$ is the soft space $\mathrm{Susp}(X):=(\Sigma X,\tau_\pi,E).$
\end{definition}
Similar to Theorem~\ref{cone:slice}, under the slice-determined soft topology, each slice of a soft suspension agrees with the classical suspension of the underlying space, as stated in the next theorem.
\begin{Theorem}
For each $e\in E$, the slice $\mathrm{Susp}(X)_e$ is homeomorphic to the classical suspension of $X_e$, if $\tau_{X}$ is slice-determined soft topology.
\end{Theorem}
\begin{proof}
This is a consequence of Theorem~\ref{prop:slice-quotient-general}.
\end{proof}

\section{Group actions and soft orbit spaces}\label{sec:soft-actions}
In analogy with the classical setting, we introduce the notion of a soft orbit space by equipping the set of orbits with the soft quotient topology induced by the orbit projection. We also introduce a soft analogue of proper discontinuity and present several related observations. We begin by defining the notion of a soft group action on a soft topological space. Throughout this section, the parameter set $E$ is fixed.
\begin{definition}\label{def:soft-action}
Let $(X,\tau_X,E)$ be a soft topological space and $G$ be a group with identity $g_{0}$. A (left) \emph{soft action} of $G$ on $(X,\tau_X,E)$ is a map $\alpha : G \times X \rightarrow X,\qquad (g,x)\mapsto g\cdot x,$ such that:
\begin{enumerate}[noitemsep,topsep=0pt,parsep=0pt]
\item $e\cdot x = x$, for all $x\in X$;
\item $(gh)\cdot x = g\cdot (h\cdot x)$, for all $g,h\in G$ and $x\in X$;
\item for each $g\in G$, the map $\alpha_g : X \rightarrow X,\qquad \alpha_g(x)=g\cdot x,$ is soft $\id_E$-continuous.
\end{enumerate}
\end{definition}
\begin{remark}\label{softcontsofthome}
    By property (2) of Definition~\ref{def:soft-action}, it follows that for each $g \in G$, the map $\alpha_{g}: X \to X$ defined by $\alpha_{g}(x) = g \cdot x$ is a soft $\id_{E}$-homeomorphism, with its inverse $\alpha_{g^{-1}}$.
\end{remark}
Given a soft action of $G$ on $(X,\tau_X,E)$, we write $G\cdot x = \{g\cdot x : g\in G\}$ for the orbit of a point $x\in X$. This determines an equivalence relation $\sim_G$ on $X$ defined by 
\[
x \sim_G y \quad\iff\quad \exists\,g\in G \text{ such that } y = g\cdot x.
\]
\begin{remark}
For a left soft action $\alpha$ of a group $G$ on $(X,\tau_X,E)$, each $g \in G$ induces a map $\widetilde{\alpha}_g$ on $SS(X,E)$ defined by $\widetilde{\alpha}_g(F,E) := \phi_{\alpha_g,\id_E}(F,E),$ which we simply denote by $g \cdot (F,E)$.
\end{remark}
We now define the notions of a soft orbit space and a properly discontinuous soft action.
\begin{definition}
Let $(X,\tau_X,E)$ be a soft space with a soft action of a group $G$. The \emph{orbit space} $X/G$ is the set of equivalence classes under the relation $\sim_G$, and the canonical projection $\pi \colon X \to X/G, \quad x \mapsto G \cdot x,$ sends each point to its orbit. Endowing $(X/G,E)$ with the soft quotient topology $\tau_\pi$ induced by $(\pi,\id_E)$ (Definition~\ref{quotient_soft}) produces a soft topological space $(X/G,\tau_\pi,E)$, called the \emph{soft orbit space} of the action.
\end{definition}
\begin{example}\label{example}
Consider the real line $(\mathbb{R}, \tau_{\mathbb{R}}, E)$ with the standard soft topology. The group $(\mathbb{Z},+)$ acts on $\mathbb{R}$ by translation: $n \cdot x = x + n, \quad n \in \mathbb{Z},\ x \in \mathbb{R}.$ The resulting soft orbit space $(\mathbb{R}/\mathbb{Z}, \tau_\pi, E)$ is soft $\id_E$-homeomorphic to the standard soft circle. The homeomorphism is induced by the classical map $\mathbb{Z} \cdot t \mapsto e^{2\pi i t}$ since, the topologies on both soft spaces are slice-determined. It follows that $(f, \id_E)$ defines a soft $\id_E$-homeomorphism. Observe that by Theorem~\ref{slicequotient}, each slice of the soft orbit space is the classical orbit space of the corresponding slice. More generally, we have the following theorem.
\end{example}
\begin{Theorem}
Let $(X,\tau_X,E)$ be a soft space with a soft action of $G$, and let $(X/G,\tau_\pi,E)$ be the associated soft orbit space. If $\tau_{X}$ is a slice-determined soft topology, then for each parameter $e\in E$:
\begin{enumerate}[noitemsep,topsep=0pt,parsep=0pt]
\item the quotient map $\pi_e : X_e \rightarrow (X/G)_e$ is the classical orbit projection with respect to the $G$-action on $X_e$;
\item the slice topology $(\tau_\pi)_e$ on $(X/G)_e$ is the classical quotient topology induced by $\pi_e$.
\end{enumerate}
\end{Theorem}
\begin{proof}
The proof follows from Theorem~\ref{prop:slice-quotient-general} and Corollary~\ref{cor:slice-quotient-topology}.
\end{proof}
We now show that, as in the classical case, the orbit space projection is an open map whenever there is a soft action of a group on a soft topological space.
\begin{Theorem}
\label{prop:orbit-soft-open}
Let $(X,\tau_X,E)$ be a soft space equipped with a soft action of a group $G$, and let $(X/G,\tau_\pi,E)$ be the associated soft orbit space. Then the canonical projection $(\pi,\id_E) : (X,\tau_X,E) \rightarrow (X/G,\tau_\pi,E)$ is a soft $\id_E$-open map.
\end{Theorem}
\begin{proof}
    Let $(F,E)$ be soft open over $X$. By Definition~\ref{opensoft}, the soft set $\phi_{\pi,\id_E}(F,E)$ over $X/G$ is soft open if its inverse soft image $\phi^{-1}_{\pi,\id_E}(\phi_{\pi,\id_E}(F,E))$ is soft open over $X$. We note the inverse image satisfies 
\[
\phi^{-1}_{\pi,\id_E}(\phi_{\pi,\id_E}(F,E)) = \bigsqcup_{g \in G} g \cdot (F,E),
\]
and each $g \cdot (F,E)$ is soft open over $X$, since $\alpha_g$ is a soft $\id_E$-homeomorphism, as observed in Remark~\ref{softcontsofthome}. Therefore, the soft set $\phi_{\pi,\id_E}(F,E)$ is soft open over $X/G$. This completes the proof.
\end{proof}

We now define soft analogue of proper discontinuity of group actions.
\begin{definition}
A soft action of $G$ on $(X,\tau_X,E)$ is called \emph{properly soft discontinuous} if, for every $x \in X$, there exists a soft open set $(U,E)$ over $X$ with $x \widetilde{\in} (U,E)$ such that 
\[
g \cdot (U,E) \,\sqcap\, (U,E) = (0,E) \quad \text{for all } g \neq g_{0}.
\]
\end{definition}
\begin{remark}
If a soft action of $G$ on $(X, \tau_X, E)$ is properly soft discontinuous, then the induced action on each slice $X_e$ is properly discontinuous in the classical sense. The converse does not hold in general, as shown in the following example.
\end{remark}
\begin{example}\label{ex:counter_proper_discontinuous}
Let $X=\mathbb{R}$ with parameter set $E=\mathbb{N}$, and let $G=\mathbb{Z}$ acts on $X$ as in Example~\ref{example}. Define a soft topology $\tau$ on $X$ by declaring $(F,E)$ open if:
\begin{enumerate}[noitemsep,topsep=0pt,parsep=0pt]
    \item $F(e)$ is open in $\mathbb{R}$ for all $e \in E$;
    \item whenever $x\widetilde{\in} (F,E)$, there exists $N \in \mathbb{N}$ such that $(x-e,x+e) \subseteq F(e)$ for all $e > N$.
\end{enumerate}

For each $e_0 \in E$, the slice topology coincides with the standard Euclidean topology, so the $\mathbb{Z}$-action is properly discontinuous on every slice. However, the action is not soft properly discontinuous. For $(F,E) \in \tau$ with $0 \widetilde{\in} (F,E)$, there exists $N\in \mathbb{N}$ such that $(-N-2,N+2)\subseteq F(N+2)$. Therefore, $0\in(N+1)\cdot F(N+2) \cap F(N+2),$ and hence $(N+1)\cdot (F,E) \sqcap (F,E) \neq (0,E)$, showing that the action fails to be properly soft discontinuous.
\end{example}
The following theorem gives a criterion for proper soft discontinuity in terms of the induced actions on the slices.
\begin{theorem}\label{conversetrue}
Let $(X, \tau_X, E)$ be a soft Hausdorff topological space equipped with a soft action of a finite group $G$. If the induced action on every slice $X_e$ is properly discontinuous in the classical sense, then the soft action of $G$ on $(X, \tau_X, E)$ is properly soft discontinuous.
\end{theorem}
\begin{proof}
Fix $x\in X$. Since the induced action on each slice $X_e$ is properly discontinuous, the action is free. As $(X,\tau_X,E)$ is soft Hausdorff, for each $g\in G\setminus\{g_{0}\}$ there exist disjoint soft open sets $(F_g,E)$ and $(F'_g,E)$ over $X$ such that $x\tilde{\in}(F_g,E)$ and $g\cdot x\tilde{\in}(F'_g,E)$. Since each action map $\alpha_g(x)=g\cdot x$ is a soft $\id_E$–continuous, it follows that there exists a soft open set $(O_g,E)$ over $X$ with $x\tilde{\in}(O_g,E)\sqsubseteq(F_g,E) \quad\text{and}\quad g\cdot(O_g,E)\sqsubseteq(F'_g,E).$ Thus, $(O_g,E)\sqcap g\cdot(O_g,E)=(0,E)$. We now consider $(U,E)=\sqcap_{g\in G\setminus\{e\}}(O_g,E)$. Clearly $x\widetilde{\in} (U,E)$ and $G$ being finite, $(U,E)$ is a soft open over $X$. For $h\in G\setminus\{g_{0}\}$ we have, 
\[
(U,E)\sqcap h\cdot(U,E)\sqsubseteq (O_h,E)\sqcap h\cdot(O_h,E)=(0,E).
\]
This shows that the soft action of $G$ on $(X,\tau_X,E)$ is properly discontinuous.
\end{proof}
\begin{remark}
Example~\ref{ex:counter_proper_discontinuous} shows that the finiteness of $G$ in Theorem~\ref{conversetrue} is essential.
\end{remark}

\section{Soft Covering Spaces and Their Connection to Soft Group Actions}\label{sec:coversoft}
In \cite[Definition~3.1]{saleem2024soft}, the notion of \emph{soft covering map} is defined as a map between collections of soft sets rather than between soft topological spaces. As a consequence, a soft open set appears as a single element of the domain, making the requirement that the restriction be a soft homeomorphism vacuous and unable to reflect the local topological structure. Therefore in this section we define the notion of soft evenly covered neighborhoods and soft covering map.
\begin{definition} \label{def:evenly-covered}
Let $(X, \tau_X, A)$ and $(Y, \tau_Y, B)$ be soft topological spaces, and let $(p,e): (X, \tau_X, A) \\\rightarrow (Y, \tau_Y, B)$ be a soft surjective continuous map. A soft open set $(G,B)$ over $Y$ is \emph{soft evenly covered} by $(p,e)$ if its inverse soft image $\phi_{p,e}^{-1}(G,B)$ can be expressed as a disjoint soft union of a collection of soft open sets $\{(F_\alpha,A)\}_{\alpha\in\Lambda}$ over $X$, such that for each $\alpha \in \Lambda$ and $a \in A$, the restricted map $p\big|_{F_\alpha(a)} \colon F_\alpha(a) \to G(e(a))$ is a homeomorphism with respect to the subspace topologies induced by $X_a$ and $Y_{e(a)}$, respectively.
\end{definition}
\begin{definition}\label{softcov}
A soft map $(p,e) : (X, \tau_X, A) \to (Y, \tau_Y, B)$ is a \emph{soft covering map} if it is soft surjective, soft continuous, and soft open, and if each point $y \in Y$ has a soft open neighborhood $(G, B)$ that is soft evenly covered by $(p,e)$. In this case, $(X, \tau_X, A)$ is referred to as a \emph{soft covering space} of $(Y, \tau_Y, B)$.
\end{definition}
\begin{remark}
In Definition~\ref{softcov}, the soft $e$-open condition is necessary, as the other three conditions do not guarantee openness, see Example~\ref{ex:not-soft-open}.
\end{remark}
\begin{example}\label{ex:not-soft-open}
Let $X=Y=\{0,1\}$ and $A=B=\{e_1,e_2\}$ with $p=\id_X$ and $e=\id_A$. Equip $X$ with the discrete soft topology $\tau^{discrete}_X$ and $Y$ with the diagonal soft topology $\tau^{\Delta}_Y$. Clearly, $(p,e)$ is soft continuous and soft surjective as well. Fix $y\in Y$ and consider the diagonal soft open set $(G,B)$ defined by $G(e_1)=G(e_2)=\{y\}$. Since $p$ and $e$ are both identity maps, so $\phi^{-1}_{p,e}(G,B)=(G,B)$. Each slice restriction $p|_{\{y\}}:\{y\}\to\{y\}$ is the identity homeomorphism, hence $(G,B)$ is soft evenly covered. Thus every point of $Y$ admits a soft evenly covered neighborhood. But, $p$ is not soft $e$-open map. Indeed, let $(F,A)$ be the soft open set over $X$ given by $F(e_1)=\{1\}$ and $F(e_2)=\emptyset$. Its soft image $(G,B)$ satisfies $G(e_1)=\{1\}$ and $G(e_2)=\emptyset$, so $(G,B)\notin\tau^{\Delta}_Y$.
\end{example}
The following theorem shows that a properly soft discontinuous action induces a soft covering map via the canonical projection onto the orbit space.
\begin{theorem}\label{properdiscontimplycovering}
Let $(X, \tau_X, E)$ be a soft topological space equipped with a soft action of a group $G$. If the action is properly soft discontinuous, then the canonical projection $(\pi, \id_E) : (X, \tau_X, E) \to (X/G, \tau_{\pi}, E)$ is a soft covering map.
\end{theorem}
\begin{proof}
We note that it suffices to verify that, for each element of $X/G$, there exists a soft open set evenly covered by $(p,e)$, in view of Theorem~\ref{prop:orbit-soft-open}. Let $y \in X/G$ be an arbitrary point, and Choose a point $x \in X$ such that $\pi(x) = y$. Since the action of $G$ is properly soft discontinuous, there exists a soft open set $(F, E)$ over $X$ such that $x \tilde{\in} (F, E)$ and $g \cdot (F, E) \sqcap (F, E) = (0, E) \quad \text{for all } g \neq g_{0}.$ Set $(H, E) = \phi_{\pi, \id_E}(F, E)$. By Theorem~\ref{prop:orbit-soft-open}, the projection $\pi$ is a soft $\id_{E}$-open map, and hence $(H, E)$ is soft open over $X/G$. Since $x \tilde{\in} (F, E)$, we have $x \in F(e)$ for all $e \in E$, which implies $y = \pi(x) \in H(e)$ for all $e \in E$. Thus $y \tilde{\in} (H, E)$. It remains to show that $(H, E)$ is soft evenly covered by $(\pi, \id_E)$. We note that 
\[
\phi_{\pi, \id_E}^{-1}(H, E) = \bigsqcup_{g \in G} g \cdot (F, E),
\]
which is soft open over $X$, since each $\alpha_g$ is a soft $\id_{E}$ homeomorphism. We claim that $g_{1}\cdot (F, E) \sqcap g_{2}\cdot (F, E) = (0, E)$ whenever $g_{1} \neq g_{2}$. Suppose this is not the case. Then there exist $e_{0} \in E$ and $x_{0} \in X$ such that, $x_{0} \in g_{1}\cdot F(e_{0}) \cap g_{2}\cdot F(e_{0}).$ Equivalently, $g_{2}^{-1} g_{1}\cdot (F, E) \sqcap (F, E) \neq (0, E)$. Since the soft action of $G$ is properly soft discontinuous, we have $g_{1} = g_{2}$, which is a contradiction. We show that, for each $g \in G$ and $e \in E$, the restriction $\pi|_{g\cdot F(e)} \colon g\cdot F(e) \to H(e)$ is a homeomorphism onto the slice. By construction, the map is surjective. To verify injectivity, let $z_{1}, z_{2} \in g\cdot F(e)$ with $\pi(z_{1}) = \pi(z_{2})$. Then $z_{2} = h \cdot z_{1}$ for some $h \in G$. Since both points lie in $g\cdot F(e)$, it follows that $g^{-1} z_{1} \in F(e)$ and $(g^{-1} h g)\cdot (g^{-1} z_{1}) \in F(e)$. As $(F, E)$ is disjoint from $g'\cdot (F, E)$ for $g' \neq g_{0}$, we obtain $g^{-1} h g = g_{0}$, hence $h = g_{0}$ and $z_{1} = z_{2}$. Since $\pi$ is soft $\id_{E}$-continuous and soft $\id_{E}$-open map, its restriction to each slice $g\cdot F(e)$ is an open by Remark~\ref{softopenusualopen}. Therefore, $\pi|_{g\cdot F(e)}$ is a homeomorphism. Thus, $(H, E)$ is a soft evenly covered neighborhood of $y$. Hence, $(\pi, \id_E)$ is a soft covering map.
\end{proof}
The converse of Theorem~\ref{properdiscontimplycovering} does not hold in general, as shown by the following example.
\begin{example}
Let $(X,\tau_X,E)$ be a soft topological space, and let $G$ be a nontrivial group acting trivially on $X$. Then each orbit is a singleton, and hence $X/G$ is canonically identified with $X$. Under this identification, the orbit projection $(\pi,\id_E)\colon (X,\tau_X,E)\rightarrow (X/G,\tau_\pi,E)$ is the identity map. Consequently, $(\pi,\id_E)$ is a soft $\id_E$-homeomorphism, and in particular a soft covering map. For any non-identity element $g\in G$ and any $(F,E)\in\tau_X$, one has $g\cdot (F,E)=(F,E)$, and therefore 
\[
g\cdot (F,E)\sqcap (F,E)=(F,E)\neq (0,E).
\]
Hence the action of $G$ is not properly soft discontinuous.
\end{example}
Although the converse of the Theorem~\ref{properdiscontimplycovering} fails in general, it holds under additional hypothesis, as shown in the following theorem.
\begin{theorem}\label{thm6.0.7}
Let $(X,\tau_X,E)$ be a soft topological space equipped with a soft action of a group $G$ and suppose that the canonical projection $(\pi,\id_E):(X,\tau_X,E)\rightarrow (X/G,\tau_\pi,E)$ is a soft covering map. If the action of $G$ is free, then it is properly soft discontinuous.
\end{theorem}
\begin{proof}
    Fix $x\in X$ and set $y:=\pi(x)\in X/G$. Since $(\pi,\id_E)$ is a soft covering map, it follows that $y$ admits a soft evenly covered neighborhood $(H,E)$ for which there exists a collection $\{(F_\alpha,E)\}_{\alpha\in\Lambda}$ of soft open sets over $X$ whose elements are whose disjoint soft union is $\phi^{-1}_{\pi,\id_E}(H,E)$ and such that for all $\alpha\in\Lambda$ and $a\in E$ the restrictions $\pi\big|_{F_\alpha(a)}\colon F_\alpha(a)\to H(a)$ are homeomorphisms. Now we choose $\alpha_0\in \Lambda$ such that with $x\tilde{\in} (F_{\alpha_0},E)$. Let $g\in G\setminus\{g_{0}\}$. If $g\cdot (F_{\alpha_0},E)\ \sqcap\ (F_{\alpha_0},E) \neq (0,E)$, then there exists $g\in G$ such that $g\cdot (F_{\alpha_0},E)\ \sqcap\ (F_{\alpha_0},E)\neq(0,E)$ and $g\neq g_{0}$. This implies that $u_2=g\cdot u_1$, where $u_1,u_2\in F_{\alpha_0}(a)$ for some $a\in E$. Thus $\pi(u_1)=\pi(u_2)$. Since $\pi\vert_{F_{\alpha_0}(a)}$ is injective, it follows that $u_1=u_2=g\cdot u_1$, which contradicts the assumption that the action of $G$ is free. Hence $g \cdot (F_{\alpha_0},E) \sqcap (F_{\alpha_0},E) = (0,E)$ for all $g \neq g_{0}$, and so the action of $G$ is properly soft discontinuous.
\end{proof}

\section{Application: Multi-Agent Motion Planning under Parametric Uncertainty}\label{application}
In algorithmic robotics, planning collision-free trajectories for multi-agent systems, such as aerial swarms or automated vehicles, is fundamentally constrained by exponential expansion of the joint configuration space with each additional agent \cite{lavalle2006planning}. While existing methodologies, including the quotient-space road map planner \cite{orthey2018quotient} and normal parameter reduction \cite{kong2008normal}, effectively exploit symmetries or reduce parameter dimensionality, they lack a rigorous topological framework to ensure the lossless preservation of collision-free regions under parametric uncertainty arising from dynamic environmental conditions, such as fluctuating weather or sensor conditions \cite{zhang2004quotient}. In this section, we illustrate a potential application of our framework, which provides a mathematically rigorous and combinatorial reduction of such configuration spaces.
\subsection{Mathematical Formulation and Combinatorial Reduction}
Consider $N$ identical agents navigating a workspace graph $V$. The classical collision-free configuration space is $X = V^N \setminus \Delta$, where the diagonal $\Delta=\{(x_1, \dots, x_N) \in V^N \mid x_i = x_j \text{ for some } i \neq j\}$ represents physical collisions \cite{kong2008normal}. The size of this search space, $|X| = \frac{|V|!}{(|V|-N)!}$. Due to the indistinguishability of the agents, the configuration space $X$ admits permutation symmetry. This symmetry is captured by the symmetric group $S_N$, which acts on $X$ by permuting the agent coordinates. To mathematically capture dynamically fluctuating collision-free regions due to external factors, we utilize soft sets. Let $E$ be a parameter set representing varying environmental conditions (for example ``Clear Weather'', ``Heavy Fog'' etc.). For each $e \in E$, let $V_e \subseteq V$ denote the subset of collision-free nodes. The parameterized collision-free space is modelled as a soft open set $(F,E) \in \tau_X$: 
\[
F(e) = \{(x_1, \dots, x_N) \in X \mid x_i \in V_e \text{ for all } i\}.\] However, as $N$ increases, the combinatorial growth of $F(e)$ induces a severe computational burden on standard algorithms. To bypass this computational barrier, we project the parameterized collision-free space $(F, E)$ onto the soft orbit space $X/S_N$. The canonical projection $\pi(x_1, \dots, x_N) = \{x_1,\dots,x_N\}$ maps an ordered tuple of configurations to an unordered set of agent positions, yielding the deeply reduced soft open set $(H,E) = \pi(F,E)$. For a fixed parameter $e$, the number of functionally distinct collision-free states decreases from the permutation count to $|H(e)| = \binom{|V_e|}{N} = \frac{|V_e|!}{N!(|V_e|-N)!}$.

This reduction shrinks the search space by a factor of $N!$. By operating directly within the soft orbit space, fundamental graph search algorithms, such as A* algorithm \cite{hart1968formal} and Dijkstra's algorithm \cite{dijkstra1959note} bypass the redundant computation, significantly accelerating trajectory generation.
\subsection{Theoretical Guarantee}
Compressing configuration spaces risks creating false pathways where disjoint collision-free regions appear connected, causing algorithms to blindly navigate into forbidden regions. Our framework mathematically eliminates this computational risk. Since the action of the symmetric group $S_N$ on the collision-free configuration space $X$ is free, Theorem~\ref{thm6.0.7} implies that the action is properly soft discontinuous; this, in turn, ensures by Theorem~\ref{properdiscontimplycovering} that the canonical projection $\pi$ is precisely a soft covering map. Consequently, every state in the compressed space $X/S_{N}$ possesses a soft evenly covered neighbourhood. This guarantees that for each parameter $e \in E$, local topological structures are preserved. Thus, any trajectory planned within the compressed slice $H(e)$ admits a unique lift to the corresponding unreduced configuration space, thereby eliminating topological distortion and spurious collisions.

\section{Conclusion}
The framework established in this paper provides a rigorous foundation for further developments in soft algebraic topology. A primary direction for future research is the formulation of Soft CW complexes, which, combined with existing work on the soft fundamental group and singular homology \cite{BahredarKouhestaniPassandidehFundamentalGroup,singularhomology}, will lead to a comprehensive theory of higher soft homotopy groups. Beyond these theoretical advancements, the symmetry reduction demonstrated in Section~\ref{application} highlights the potential application of soft topological methods to address computational challenges, such as multi-agent motion planning under parametric uncertainty. Future research could explore extending these soft quotient and covering frameworks to soft topologies over fuzzy sets, further broadening their applicability. Furthermore, we aim to explore the applicability of our framework to topological data analysis.

\begin{acknow}
The first author is grateful for financial support in the form of Prime Minister’s Research Fellowship, Government of India (PMRF/2502403). The second author was supported by the Centre for Operator Algebras, Geometry, Matter and Spacetime, Ministry of Education, Government of India through Indian Institute of Technology Madras (Project no. SB22231267MAETWO008573).
\end{acknow}
\bibliographystyle{amsplain}  
\bibliography{reference}
\end{document}